\documentclass[12 pt]{amsart}
\usepackage{full page}
\usepackage{graphicx}
\usepackage{amsmath}
\usepackage{amssymb}
\usepackage{mathrsfs}
\usepackage{bbm}
\usepackage{tikz-cd}
\usepackage{enumerate}
\usepackage{url}
\usepackage{amsthm}
\usepackage{mathdots}
\definecolor{mypink}{cmyk}{0, 0.7808, 0.4429, 0.1412}
\definecolor{lightblue}{cmyk}{1,.4,.4,.1}
\definecolor{myred}{cmyk}{0,.82,.87,.25}
\definecolor{myblue}{cmyk}{.81,.41,0,.09}
\usepackage[colorlinks = true,
	linkcolor = black,
	urlcolor = mypink,
	citecolor = myblue]{hyperref}
\usetikzlibrary{cd}
\title{Universal additive Chern classes and a GRR-type theorem}
\author{Eoin Mackall}
\address{Mathematical \& Statistical Sciences, University of Alberta, Edmonton, CANADA}
\email{mackall \emph{at} ualberta.ca}
\urladdr{\url{www.ualberta.ca/~mackall}}
\date{\today}
\newtheorem{thm}{Theorem}[section]

\newtheorem{prop}[thm]{Proposition}
\newtheorem{cor}[thm]{Corollary}
\newtheorem{lem}[thm]{Lemma}

\theoremstyle{definition}
\newtheorem{defn}[thm]{Definition}
\newtheorem{exmp}[thm]{Example}
\newtheorem{rmk}[thm]{Remark}
\begin{document}
\begin{abstract}
We construct a functor, from the category of schemes to the category of graded rings, that is an initial object for having a theory of Chern classes with an additive first Chern class. For any scheme $X$, the graded ring that our functor associates to $X$ is related to the associated graded ring of the $\gamma$-filtration on the Grothendieck ring of finite rank locally free sheaves on $X$ via a Grothendieck-Riemann-Roch type theorem.
\end{abstract}
\maketitle
\noindent\textbf{Conventions}. We fix an arbitrary base field $k$. A variety over $k$, or simply a variety when the field $k$ is clear from context, is a separated scheme of finite type over $k$.

In this text we write $\mathbb{Z}_{\geq 0}=\{0,1,2,\ldots\}$ for the set of all nonnegative integers. When we say that a ring is graded, we mean that it is a $\mathbb{Z}_{\geq 0}$-graded ring.
\section{Introduction}
A number of well-studied functors throughout algebraic geometry are equipped with a satisfactory notion of Chern classes. Classical examples of such functors include the even $\ell$-adic cohomology $\bigoplus_{i}\mathrm{H}^{2i}(-,\mathbb{Z}_\ell(i))$ for smooth and quasiprojective varieties defined over fields of characteristic $p\neq \ell$ and the Chow ring $\mathrm{CH}(-)$ of cycles modulo rational equivalence for smooth varieties over arbitrary base fields. A lesser known example is the associated graded ring for the $\gamma$-filtration on the Grothendieck ring $K(-)$ of finite rank locally free sheaves; this example is the most general, however, and applies for any noetherian scheme.

In this paper, we formalize and uniformize what it means for a functor to have a theory of Chern classes where the first Chern class is additive. We do this by producing a contravariant functor $B(-)$, from any subcategory of the category of schemes that is closed under taking projective bundles to the category of graded rings, that is the universal receptor of Chern classes in a proper sense, see Proposition \ref{up}.

Surprisingly, the functor $B(-)$ that we construct is most naturally related to the associated graded ring $\mathrm{gr}_\gamma K(-)$ for the $\gamma$-filtration on the ring $K(-)$. In Theorem \ref{grr}, we make this precise by showing that there is a Grothendieck-Riemann-Roch (GRR)  without denominators type relation that holds between these rings. As an immediate consequence to this theorem we give Corollary \ref{cor} that says the functor $B(-)$ calculates the Chow ring of a smooth surface with integral coefficients and the Chow ring of an arbitrary smooth variety rationally.

This paper consists of three sections. The first section focuses on the construction of the functor $B(-)$ and on some of its functorial properties. The second section is more concrete; the focus of this section is on proving specific formula between Chern classes and on outlining techniques that one can use to study the ring $B(X)$ for a given scheme $X$. Our last section is devoted to the main result of this paper, Theorem \ref{grr}.\\

\noindent\textbf{Acknowledgements}. I'd like to thank an anonymous referee for giving me the motivation to revisit this text. I would also like to thank this referee for suggesting that I look again at the source \cite{MR0354655}. I think that these two suggestions greatly improved the exposition and content of the original version of this paper.

\section{Construction and fundamental properties}\label{cfp}
Throughout this section, we let $X$ be an arbitrary scheme. We use the notation $K(X)$ for the Grothendieck ring of finite rank locally free sheaves on $X$. For any morphism of schemes $f:X\rightarrow Y$, we write $f^*:K(Y)\rightarrow K(X)$ for the morphism induced by the pullback of sheaves of modules. Recall \cite{MR0265355} that the assignment \[X\rightsquigarrow K(X) \quad \mbox{and}\quad f\rightsquigarrow f^*\] defines a contravariant functor $K:\mathsf{Sch}\rightarrow \mathsf{Ring}$ from the category of schemes to the category of (commutative) rings. Our goal, for the rest of this section, is to construct a graded ring $B(X)=\bigoplus_{i\geq 0} B^i(X)$ that acts as a universal receptor of Chern classes from $K(X)$. More precisely, we aim to construct a contravariant functor $B:\mathsf{Sch}\rightarrow \mathsf{GrRing}$ and a collection of natural transformations $c_i^B:K\rightarrow B^i$ (considered as functors to the category $\mathsf{Set}$ of sets) having the following universal property.

\begin{prop}\label{up}
	Let $A:\mathsf{Sch}\rightarrow \mathsf{GrRing}$ be any contravariant functor and suppose that, for every integer $i\geq 0$, there exists a natural transformation $c_i^A:K\rightarrow A^i$, of functors from $\mathsf{Sch}$ to $\mathsf{Set}$, having the following properties:
	\begin{enumerate}[\normalfont (1)]
		\item for all finite rank locally free sheaves $\mathcal{F}$ on $X$ we have $c_0^A(\mathcal{F})=1$;
		\item the term $c_i^A(\mathcal{F})=0$ is vanishing for all integers $i$ greater than the rank $\mathrm{rk}(\mathcal{F})$ of $\mathcal{F}$;
		\item for any short exact sequence of finite rank locally free sheaves $0\rightarrow \mathcal{E}\rightarrow \mathcal{G}\rightarrow \mathcal{F}\rightarrow 0$ there is a Whitney sum relation $c_i^A(\mathcal{G})=\sum_{j=0}^i c_{i-j}^A(\mathcal{F})c_j^A(\mathcal{E})$;
		\item for any pair of invertible sheaves $\mathcal{L},\mathcal{L}'$ there is a relation $c_1^A(\mathcal{L}\otimes \mathcal{L}')=c_1^A(\mathcal{L})+c_1^A(\mathcal{L}')$;
		\item for any finite rank locally free sheaf $\mathcal{E}$ over $X$ with projective bundle $P=\mathbb{P}(\mathcal{E})\rightarrow X$, the pullback map $A(X)\rightarrow A(P)$ is injective.
	\end{enumerate}
	Then there exists a natural transformation $b_A:B\rightarrow A$ of functors which is completely determined by the rule $b_A(c_i^B([\mathcal{F}]))= c_i^A([\mathcal{F}])$.
\end{prop}

We write $R_X$ for the set of symbols $\{c_i^B(\mathcal{F})\}$ varying over all integers $i\geq 0$ and over all finite rank locally free sheaves $\mathcal{F}$ on $X$. The algebra $\mathbb{Z}[R_X]$ generated by these symbols is naturally graded with each symbol $c_i^B(\mathcal{F})$ having degree-$i$. Define $I_X\subset \mathbb{Z}[R_X]$ to be the ideal having generators:
\begin{enumerate}[\textperiodcentered]
\item $c_0^B(\mathcal{F})-1$ for all $\mathcal{F}$
\item $c_i^B(\mathcal{F})$ for all integers $i>\mathrm{rk}(\mathcal{F})$
\item $c_i^B(\mathcal{G})-\sum_{j=0}^i c_{i-j}^B(\mathcal{F})c_j^B(\mathcal{E})$ for any short exact sequence $0\rightarrow \mathcal{E}\rightarrow \mathcal{G}\rightarrow \mathcal{F}\rightarrow 0$
\item $c_1^B(\mathcal{L}\otimes \mathcal{L}')-c_1^B(\mathcal{L})-c_1^B(\mathcal{L}')$ for any pair of invertible sheaves $\mathcal{L},\mathcal{L}'$.
\end{enumerate} We denote by $[c_i^B(\mathcal{F})]$ the class of $c_i^B(\mathcal{F})$ in $\mathbb{Z}[R_X]/I_X$. Note that $I_X$ is a homogeneous ideal so that the quotient $\mathbb{Z}[R_X]/I_X$ is graded with a well-defined notion of degree.

For any other scheme $Y$ and for any morphism $f:X\rightarrow Y$ there is a natural map $$f^*:\mathbb{Z}[R_Y]/I_Y\rightarrow \mathbb{Z}[R_X]/I_X$$ defined by $f^*[c_i^B(\mathcal{F})]:=[c_i^B(f^*\mathcal{F})]$.

We write $\mathbb{P}^X$ for the directed set of maps to $X$ made up of chains of projective bundles. By this we mean $\mathbb{P}^X$ is the set of finite sequences of maps $$X\leftarrow P_1\leftarrow P_2\leftarrow \cdots$$ where $P_1\rightarrow X$ is a composition of projections from successive projective bundles, $P_2\rightarrow P_1$ is likewise a chain of projective bundles over $P_1$, and so on. One chain dominates another chain if there is a commutative ladder with each vertical arrow a chain of projective bundles. \[\begin{tikzcd} X & \arrow{l} P_1 & \arrow{l} P_2 & \arrow{l} \cdots\\
X\arrow[equals]{u} & \arrow{l}\arrow{u} Q_1 & \arrow{l}\arrow{u}Q_2 & \arrow{l} \cdots
\end{tikzcd}\] In the above diagram the bottom sequence, call it $S_Q$, dominates the top, $S_P$, and we would write $S_Q\geq S_P$. Any two chains have a chain that dominates them. To see this, let $$X\leftarrow P_1\leftarrow P_2\leftarrow \cdots$$ $$X\leftarrow Q_1\leftarrow Q_2\leftarrow \cdots$$ be two such chains. Then, by taking fiber products, a third such chain that dominates the two given is $$X\leftarrow P_1\times_X Q_1\leftarrow P_2\times_{P_1\times_X Q_1} Q_2\leftarrow \cdots.$$  

A chain $P$ of chains of projective bundles $$X\leftarrow P_1\leftarrow P_2\leftarrow \cdots$$ determines a directed system using the natural maps defined above $$\mathbb{Z}[R_X]/I_X\rightarrow \mathbb{Z}[R_{P_1}]/I_{P_1}\rightarrow \mathbb{Z}[R_{P_2}]/I_{P_2}\rightarrow \cdots.$$ Denoting the limit of this directed system by $\mathbb{Z}[R^\infty_P]=\varinjlim \mathbb{Z}[R_{P_i}]/I_{P_i}$, we get a directed system of the rings $\mathbb{Z}[R^\infty_P]$ over all chains $P$ in the set $\mathbb{P}^X$.

\begin{defn} We define the ring $B(X)$ as the quotient $$\mathbb{Z}[R_X]/\mathrm{ker}(f_X)$$ where $f_X:\mathbb{Z}[R_X]\rightarrow \varinjlim_{P\in \mathbb{P}^X} \mathbb{Z}[R_P^\infty]$ is the canonical map.
\end{defn}

From now on we write $[c_i^B(\mathcal{F})]$ only for the class of $c_i^B(\mathcal{F})$ in $B(X)$. The ring $B(X)$ inherits a canonical grading where, for any integer $i\geq 0$, the subgroup $B^i(X)\subset B(X)$ that consists of degree-$i$ elements is generated by those products $$[c_{i_1}^B(\mathcal{F}_1)]\cdots [c_{i_n}^B(\mathcal{F}_n)]$$ of degree $i_1+\cdots +i_n= i$. Additionally, for any other scheme $Y$ and for any morphism $f:X\rightarrow Y$ there is an induced morphism of graded rings $$f^*:B(Y)\rightarrow B(X)$$ uniquely determined by the formula $f^*[c_i^B(\mathcal{F})]=[c_i^B(f^*\mathcal{F})]$ as before. The assignment \[X\rightsquigarrow B(X)\quad \mbox{and} \quad f\rightsquigarrow f^*\] thus gives rise to a contravariant functor $B:\mathsf{Sch}\rightarrow \mathsf{GrRing}$.

We write $1+tB(X)[[t]]$ for the group of invertible formal power series in a variable $t$ with coefficients in $B(X)$. The following definition is key to Proposition \ref{up}.

\begin{defn}
	For any scheme $X$, the total Chern class is the homomorphism $$c^B_t:K(X)\rightarrow 1+tB(X)[[t]]$$ that's uniquely determined by the formula $$c^B_t(x)= 1+[c^B_1(\mathcal{F})]t+[c^B_2(\mathcal{F})]t^2+\cdots$$ whenever $x=[\mathcal{F}]$ is the class of a locally free sheaf $\mathcal{F}$.
	
	For any $i\geq 0$, we define the $i$th Chern class as the natural transformation $c_i^B:K\rightarrow B^i$ induced by taking the degree-$i$ component of the total Chern class homomorphism.
\end{defn}

It's not difficult to show that the functor $B$ and the $i$th Chern classes $c_i^B$ have the universal property described in Proposition \ref{up}; we leave this to the reader.

\begin{rmk}\label{catchall}
Often times, the functor $B$ will be too restrictive to be of much use. What we mean by this is that, it will often be beneficial to consider an analog of $B$ that is universal in the sense of Proposition \ref{up} but, only among all functors coming from a strictly smaller category $\mathsf{C}$ than the category $\mathsf{Sch}$ of all schemes. Such a functor can be constructed by considering in our construction of $B$, and in the universal property of Proposition \ref{up}, only those chains of projective bundles that exist inside $\mathsf{C}$.

Denote by $B_{\mathsf{C}}$ the functor constructed in this way that is universal among all such functors coming from $\mathsf{C}$. Assume the category $\mathsf{C}$ consists exactly of those schemes having a property $\mathcal{P}$ that is closed in the sense: if a scheme $X$ with property $\mathcal{P}$ has a projective bundle $P\rightarrow X$, then the scheme $P$ also has property $\mathcal{P}$. For example, $\mathcal{P}$ could be the property of being a variety, or of being noetherian, or of being a regular variety. Then, in this setting, the rings $B(X)$ and $B_{\mathsf{C}}(X)$ will be isomorphic for any scheme $X$ contained in $\mathsf{C}$.
\end{rmk}

The rest of this section is devoted to studying those properties of the functor $B$ that are most similar to properties of a cohomology theory. Although the functor $B$ turns out not to be a cohomology theory itself (even when restricted to the category of smooth and projective varieties), it does share a number of properties that are typical of a cohomology theory (e.g.\ homotopy invariance and continuity).

\begin{lem}\label{keylem}
Let $X$ and $Y$ be schemes and let $f:X\rightarrow Y$ be a morphism between them. Assume that $f^*:K(Y)\rightarrow K(X)$ is surjective. Then $f^*:B(Y)\rightarrow B(X)$ is surjective. 
\end{lem}

\begin{proof}
It suffices to show that each class $[c_i^B(\mathcal{F})]$ is in the image of $f^*$ as $\mathcal{F}$ ranges over all finite rank locally free sheaves on $X$, and $i$ ranges over all integers $i\geq 0$. Since the following diagram commutes for any $i\geq 0$,
\[\begin{tikzcd}K(Y)\arrow{r}{f^*}\arrow{d}{c_i^B} & K(X)\arrow{d}{c^B_i}\\ B^i(Y)\arrow{r}{f^*} & B^i(X) \end{tikzcd}\] the lemma follows from observing that there is a class $x$ in $K(Y)$ mapping to $\mathcal{F}$ under $f^*$.
\end{proof}

\begin{lem}\label{homotopy}
Let $X$ be a scheme and let $E$ be a vector bundle over $X$ of finite rank with structure map $\pi:E\rightarrow X$. Assume the pullback $\pi^*:K(X)\rightarrow K(E)$ is an isomorphism. Then $\pi^*:B(X)\rightarrow B(E)$ is an isomorphism.
\end{lem}

\begin{proof}
Let $\sigma:X\rightarrow E$ be the zero section of $E$ so that the composite $\pi\circ \sigma$ is the identity of $X$. By functorality $\sigma^*\circ \pi^*$ is the identity on $B(X)$ and the map $\pi^*$ is therefore injective. Surjectivity of $\pi^*$ follows from Lemma \ref{keylem}.
\end{proof}

\begin{lem}
Suppose that $X$ is regular, noetherian, and has an ample invertible sheaf. Then for any open subscheme $U\subset X$ with inclusion $i:U\rightarrow X$, the restriction $i^*:B(X)\rightarrow B(U)$ is surjective.
\end{lem}

\begin{proof}
Under the given conditions, the pullback $i^*:K(X)\rightarrow K(U)$ is a surjection \cite{MR0265355}. The claim is then immediate from Lemma \ref{keylem}.
\end{proof}

\begin{lem}
Let $x$ be a point of $X$. Then there are isomorphisms
$$\varinjlim_{x\in U} B(U)=B(\varprojlim_{x\in U} U)=B(\mathrm{Spec}(\mathcal{O}_{X,x}))=\mathbb{Z}$$ where the limits are taken along open subschemes $U$ containing $x$ with respect to inclusions.
\end{lem}

\begin{proof} 
	Since projective modules over a local ring are free, we have $K(\mathrm{Spec}(\mathcal{O}_{X,x}))=\mathbb{Z}$ with generator the class of $\mathcal{O}_{X,x}$. Thus, the canonical map $K(U)\rightarrow K(\mathrm{Spec}(\mathcal{O}_{X,x}))$ is surjective for every open subscheme $U\subset X$ containing $x$ and the surjectivity of the canonical map $$\varinjlim_{x\in U} B(U)\rightarrow B(\varprojlim_{x\in U} U)$$ follows from Lemma \ref{keylem}.

To show injectivity of this map, it suffices to show every Chern class of positive degree is trivial over some open set around $x$. But, this is true for every vector bundle on $X$ so that it is also true for every Chern class.
\end{proof}
\section{Chern classes and $\lambda$-rings}\label{cclr}
The functor $B$ defined in the previous section is compatible, in some way, with the $\lambda$-ring structure of the Grothendieck ring. For most, if not all, schemes $X$ where we can describe the ring $B(X)$, we depend heavily on this compatibility. We therefore take the time, in this section, to develop some of these relations, between the functor $B$ and the $\lambda$-ring structure of the Grothendieck ring, more precisely.

Recall (\cite[Expose 0]{MR0354655}, \cite{MR0265355}, or \cite{MR801033})  that for a scheme $X$ the Grothendieck ring $K(X)$ is equipped with a canonical structure of a $\lambda$-ring. This means that for any integer $i\geq 0$ there are natural transformations (of functors to $\mathsf{Set}$) $\lambda^i:K\rightarrow K$ defined so that $\lambda^i([\mathcal{F}])=[\Lambda^i(\mathcal{F})]$ for any sheaf $\mathcal{F}$. These natural transformations have the properties:
\begin{enumerate}[\normalfont (1)]
	\item $\lambda^0(x)=1$ for all $x$ in $K(X)$
	\item $\lambda^1(x)=x$ for all $x$ in $K(X)$
	\item $\lambda^i(x+y)=\sum_{j=0}^i \lambda^{i-j}(x)\lambda^{j}(y)$
	\item $\lambda^i(xy)=P_i(\lambda^1(x),...,\lambda^i(x),\lambda^1(y),...,\lambda^i(y))$ for certain universal polynomials $P_i$
	\item $\lambda^i(\lambda^j(x))=P_{i,j}(\lambda^1(x),...,\lambda^{ij}(x))$ for certain universal polynomials $P_{i,j}$.
\end{enumerate}

\begin{rmk}
For any $\lambda$-ring $R$, there are well-defined Schur operations $S^\mu:R\rightarrow R$ for any partition $\mu=(\mu_1,...,\mu_n)$ defined by $$S^\mu(x)=\mathrm{det}(\lambda^{\mu_i+j-i}(x))_{1\leq i,j \leq n}.$$ If $\epsilon\subset \mu$ is another partition, one can define an operation $S^{\mu/\epsilon}:R\rightarrow R$ for the skew diagram $\mu/\epsilon$ as the sum
$$S^{\mu/\epsilon}(x)=\sum_{\nu} c_{\epsilon, \nu}^\mu S^\nu(x)$$ where $c_{\epsilon,\nu}^\mu$ is a Littlewood-Richardson coefficient. These operations satisfy the formula $$S^{\mu/\epsilon}(x+y)=\sum_{\epsilon\subset \nu\subset \mu}S^{\nu/\epsilon}(x) S^{\mu/\nu}(y)$$ generalizing that for the $\lambda$-operations.
\end{rmk}

\begin{exmp}\label{gengrass}
Let $\mathrm{Gr}(m,n)$ be the Grassmannian variety of $m$-dimensional planes in an $n$-dimensional $k$-vector space. Then $K(\mathrm{Gr}(m,n))$ is additively generated by the classes $S^\mu(Q)$ where $Q$ is the universal quotient sheaf on $\mathrm{Gr}(m,n)$ of rank $n-m$ and $\mu$ ranges over partitions which fit inside a box of size $(n-m)\times m$.
\end{exmp}

Together Lemma \ref{sym} and Lemma \ref{gen} below show that a collection of $\lambda$-ring generators for the Grothendieck ring $K(X)$ of a scheme $X$ determine a collection of generators for $B(X)$.

\begin{lem}\label{sym}
	For any locally free sheaves $\mathcal{F},\mathcal{G}$ of ranks $n,m$ respectively and for any $i\geq 1$, there exist polynomials $Q_{n,m,i}$ so that $$c_t^B(\mathcal{F}\otimes \mathcal{G})=1+\sum_{i\geq 1} Q_{n,m,i}(c_1^B(\mathcal{F}),...,c_i^B(\mathcal{F}),c_1^B(\mathcal{G}),...,c_i^B(\mathcal{G}))t^i$$ inside of $1+tB(X)[[t]]$.
\end{lem}

\begin{proof}
	It suffices to work over a chain of projective bundles $P$ where the classes of $\mathcal{F},\mathcal{G}$ split into a sum of invertible sheaves in $K(P)$. Assuming this is the case, we can write $[\mathcal{F}]=[\mathcal{L}_1]+\cdots + [\mathcal{L}_n]$, for some invertible sheaves $\mathcal{L}_1,...,\mathcal{L}_n$, and $[\mathcal{G}]=[\mathcal{L}'_1]+\cdots + [\mathcal{L}'_m]$, for some invertible sheaves $\mathcal{L}_1',..., \mathcal{L}_m'$, so that $$c_t^B(\mathcal{F}\otimes \mathcal{G})=\prod_{1\leq i\leq n,1\leq j\leq m}(1+c^B_1(\mathcal{L}_i)+c^B_1(\mathcal{L}'_j)).$$ Since the product on the right-hand-side of this equality is symmetric in the Chern classes $c^B_1(\mathcal{L})$'s and symmetric in the Chern classes $c^B_1(\mathcal{L}')$'s, the claim follows by choosing $Q_{n,m,i}$ to be the homogeneous polynomial expressing the weight $i$ part of this product as a polynomial of in elementary symmetric polynomials in these variables.
\end{proof}

\begin{exmp}[c.f.\ {\cite[Example 3.2.2]{MR1644323}}]\label{line}
	If $\mathcal{F}$ is a locally free sheaf of rank $n$ and if $\mathcal{L}$ is an invertible sheaf then, for any $j\geq 1$, $$c^B_j(\mathcal{F}\otimes \mathcal{L})=\sum_{i=0}^j\binom{n-i}{j-i}c^B_i(\mathcal{F})c^B_1(\mathcal{L})^{j-i}.$$ Equivalently, $$c_t^B(\mathcal{F}\otimes \mathcal{L})=c_t^B(\mathcal{L})^nc^B_\tau(\mathcal{F})$$ where $\tau=t/c_t^B(\mathcal{L})$.
\end{exmp}

\begin{lem}\label{gen}
	For any finite rank locally free sheaf $\mathcal{F}$, the Chern class $c^B_j(\lambda^i(\mathcal{F}))$ can be expressed as a polynomial in the Chern classes of $\mathcal{F}$.
\end{lem}

\begin{proof}
	Let $x=[\mathcal{F}]$. We work over a chain of projective bundles $P$ where the class $x$ splits $x=x_1+\cdots + x_n$ into a sum of classes of some invertible sheaves $x_1,...,x_n$. Then $$c^B_t(\Lambda^j\mathcal{F})=\prod_{i_1<\cdots < i_j} (1+c^B_1(x_{i_1})+\cdots+c^B_1(x_{i_j}))$$ is symmetric in the $c^B_1(x_k)'s$ which proves the claim.
\end{proof}

\begin{rmk}[c.f.\ {\cite[Remark 3.2.3 (a)]{MR1644323}}]\label{dual}
	The same process allows one to determine the Chern classes of the dual $\mathcal{F}^\vee$ of a finite rank locally free sheaf $\mathcal{F}$. Explicitly, $$c_j^B(\mathcal{F}^\vee)=(-1)^jc^B_j(\mathcal{F})$$ for any integer $j\geq 0$.
\end{rmk}

These observations were the motivation for the definition:

\begin{defn}
	We define the \textit{level} of a variety $X$
	\[ \mathrm{lev}(X)=\min\{ \#S : S\subset K(X) \text{ generates $K(X)$ as a $\lambda$-ring}\} \]
	as the minimal cardinality of a finite subset $S\subset K(X)$ such that the elements of $S$ generate $K(X)$ as a $\lambda$-ring. If there are no finite subsets $S\subset K(X)$ that generate $K(X)$, then the level is said to be infinite, i.e.\ $\mathrm{lev}(X)=\infty$.
\end{defn}

\begin{exmp}
	If $X=\mathbb{P}^n$ or, more generally if $X=\mathrm{Gr}(m,n)$, then $\mathrm{lev}(X)=1$. 
	
	More generally still, for any sequence $0=n_0< n_1 <\cdots <n_k=n$, let $X(n_0,...,n_k)$ be the variety of $(n_0,...,n_k)$-flags in a vector space of dimension $n$. Then $\mathrm{lev}(X(n_0,...,n_k))=k-1$. To see this, note that there are $k$ tautological locally free sheaves which generate the ring $K(X)$, of finite ranks $n_1,...,n_k$ respectively. Note also that there is exactly one linear relation between these sheaves. Thus $\mathrm{lev}(X)\leq k-1$. Conversely, $\mathrm{lev}(X)\geq \mathrm{rk}_\mathbb{Z}\mathrm{Pic}(X)$ and the latter of these equals $k-1$ as well. 
	
	If $X$ is the Severi--Brauer variety associated to a central simple algebra $A$ with index $\mathrm{ind}(A)=p^n$, for some prime $p$ and some $n\geq 1$, then the level of $X$ is bounded above by the length of a subsequence of the sequence of indices of various tensor powers of $A$, see \cite[Lemma A.6]{KE}. It's likely, but difficult to show, that this upper bound is sharp.
\end{exmp}

At this point we've developed sufficiently enough theory to give a concrete example:

\begin{exmp}\label{grass} By Example \ref{gengrass}, the ring $K(\mathrm{Gr}(m,n))$ is generated by polynomials in the $\lambda$-operations of the universal quotient sheaf $Q$. By Lemmas \ref{sym} and \ref{gen}, this means that the ring $B(\mathrm{Gr}(m,n))$ is generated by the Chern classes of $Q$, call them $c_1,...,c_{n-m}$.
	
	We get relations in $B(\mathrm{Gr}(m,n))$ from the universal exact sequence of the universal sub sheaf $S$ and the universal quotient sheaf $Q$, $$0\rightarrow S\rightarrow \mathcal{O}_{\mathrm{Gr}(m,n)}^{\oplus n}\rightarrow Q\rightarrow 0.$$ If $m\leq n-m$, then from this exact sequence we find $c_t(S)=1/c_t(Q)$. Let $f_{m+1},...,f_{n}$ be the polynomials in the Chern classes of $Q$ which are the coefficients of $t^{m+1},...,t^n$ in the expansion of $1/c_t(Q)$ as a power series in $t$. If $m> n-m$, then let us rename $c_1,...,c_m$ to be the Chern classes of $S$, which evidently also generate $B(\mathrm{Gr}(m,n))$ due to the exact sequence above, and name $f_{n-m+1},...,f_{n}$ to be the coefficients of $t^{n-m+1},...,t^n$ in the expansion of $1/c_t(S)$ as a power series in the variable $t$.
	
	We claim that \[B(\mathrm{Gr}(m,n))=\begin{cases} \mathbb{Z}[c_1,...,c_{n-m}]/(f_{m+1},...,f_n) & \mbox{ if } m\leq n-m \\ \mathbb{Z}[c_1,...,c_m]/(f_{n-m+1},...,f_n) & \mbox{ if } m>n-m.\end{cases}\] Indeed, to complete the proof it's sufficient to find a cohomology theory $A$ with a theory of Chern classes in the sense of Proposition \ref{up} such that $A(\mathrm{Gr}(n,m))$ is the desired ring. Taking $A=\mathrm{CH}$ suffices (see \cite[Theorem 5.26]{MR3617981}).
\end{exmp}

The natural transformations $\lambda^i$ can also be used to define $\gamma$-operations that will play an important role later in this paper. The $\gamma$-operations are natural transformations $\gamma^i:K\rightarrow K$ determined by the pointwise formula $$\gamma^i(x)=\lambda^i(x+i-1)$$ for any element $x\in K(X)$ and for any predetermined scheme $X$. 

In this paper, we'll be particularly interested in the $\gamma$-filtration on $K(X)$ for an arbitrary scheme $X$. Recall that the $\gamma$-filtration is the ascending multiplicative filtration $F_\gamma^\bullet(X)$ of $K(X)$ defined termwise as follows: \begin{enumerate}[\textperiodcentered]
\item $F_\gamma^0(X)=K(X)$ 
\item $F_\gamma^1(X)=\mathrm{ker}(\mathrm{rk})$ where $\mathrm{rk}:K(X)\rightarrow \mathbb{Z}$ is the rank homomorphism 
\item $F_\gamma^i(X)$ is defined to be the ideal generated by monomials $\gamma^{i_1}(x_1)\cdots \gamma^{i_j}(x_j)$ where $x_1,...,x_j$ are elements of $F_\gamma^1(X)$ and $i_1+\cdots + i_j\geq i$.
\end{enumerate}

We write $\mathrm{gr}_\gamma K(X)$ for the graded ring associated to the $\gamma$-filtration of $K(X)$. By definition, the associated graded ring is the sum of quotients \[\mathrm{gr}_\gamma K(X)=\bigoplus_{i\geq 0} \mathrm{gr}_\gamma^iK(X) \quad \mbox{where} \quad \mathrm{gr}_\gamma^i K(X) =F^i_\gamma(X)/F^{i+1}_\gamma (X).\] There are a collection of Chern class functors $c_i^\gamma: K\rightarrow \mathrm{gr}_\gamma^i K$ determined for a scheme $X$ by the rule $c_i^\gamma(\mathcal{F})=\gamma^i(\mathrm{rk}(\mathcal{F})-[\mathcal{F}^\vee])$ for any locally free sheaf $\mathcal{F}$ on $X$. By the results of \cite[Chapter 5, \S2]{MR801033}, these functors satisfy the conditions (1)-(5) of Proposition \ref{up} with the caveat that they are restricted to the category $\mathsf{C}$ of all noetherian schemes. Therefore, because of Remark \ref{catchall}, there is a natural transformation $b_\gamma: B\rightarrow \mathrm{gr}_\gamma K$ of functors from the category of noetherian schemes that takes Chern classes to Chern classes. In particular, all of the formulae proved in this section that involve Chern classes hold additionally for the ring $\mathrm{gr}_\gamma K(X)$ for any noetherian scheme $X$.

\begin{rmk}\label{surj}
For any noetherian scheme $X$, the associated morphism $$B(X)\rightarrow \mathrm{gr}_\gamma K(X)$$ is a surjection since $\mathrm{gr}_\gamma K(X)$ is evidently generated by Chern classes.
\end{rmk}

We conclude this section by showing that the total Chern class of any element $x\in K(X)$ is a polynomial under some reasonable assumptions on $X$. The proof that we give here relies on a number of the examples developed in this section. We don't know of any scheme $X$ for which the following proposition doesn't hold.

\begin{prop}\label{nilp}
	Suppose that $X$ is a variety with an ample invertible sheaf $\mathcal{L}$. Then the Chern class $c^B_i(\mathcal{F})$ is nilpotent for any finite rank locally free sheaf $\mathcal{F}$ on $X$ and for any integer $i\geq 1$. In particular, the total Chern class $c^B_t(x)$ is a polynomial in $t$ for any element $x$ in $K(X)$.
\end{prop}

\begin{proof}
	If $\mathcal{F}$ is globally generated, then there is a morphism $f:X\rightarrow \mathrm{Gr}(m,n)$ for some $m,n$ such that $f^*Q=\mathcal{F}$. Since the Chern classes of $Q$ are nilpotent due to Example \ref{grass}, the same follows for the Chern classes of this $\mathcal{F}$.
	
	In the general case, since $X$ has an ample invertible sheaf $\mathcal{L}$, there is some integer $n>0$ so that both $\mathcal{F}\otimes \mathcal{L}^{\otimes n}$ and $\mathcal{L}^{\otimes n}$ are globally generated. By Example \ref{line} and induction, any Chern class of $\mathcal{F}$ can be written as a polynomial in the Chern classes of $\mathcal{F}\otimes \mathcal{L}^{\otimes n}$ and $\mathcal{L}^{\otimes n}$. Since both of these sheaves are globally generated, their Chern classes are nilpotent and thus so are the Chern classes of $\mathcal{F}$.
	
	For the final statement, we write $x=[\mathcal{F}]-[\mathcal{G}]$ and observe that for sufficiently large Chern classes of $x$ there are sufficiently large powers of the Chern classes of $\mathcal{F}$ or $\mathcal{G}$ involved. Eventually then these terms must vanish. 
\end{proof}

\section{A GRR without denominators type theorem}
Throughout this section $X$ is an arbitrary but fixed noetherian scheme. Our primary goal is to prove the following Grothendieck-Riemann-Roch without denominators type theorem between the graded ring $B(X)$, constructed in Section \ref{cfp}, and the graded ring $\mathrm{gr}_\gamma K(X)$ associated to the filtration $F_{\gamma}^\bullet(X)\subset K(X)$, that was considered near the end of Section \ref{cclr}.

\begin{thm}\label{grr}
Let $X$ be a connected noetherian scheme and write $$b^i_\gamma: B^i(X)\rightarrow \mathrm{gr}^i_\gamma K(X)$$ for the $i$th summand of the canonical morphism of Proposition \ref{up} applied to $(\mathrm{gr}_\gamma K, c_i^\gamma)$. Then the Chern classes $c_i^B$ induce well-defined maps $$c_i^B:\mathrm{gr}^i_\gamma K(X)\rightarrow B^i(X)$$ such that the compositions \[c_i^B\circ b^i_\gamma=(-1)^{i-1}(i-1)!\quad \mbox{and}\quad b^i_\gamma\circ c_i^B=(-1)^{i-1}(i-1)!\] are both multiplication by $(-1)^{i-1}(i-1)!$ for all $i>0$.
\end{thm}

The proof of Theorem \ref{grr} can be reduced to an essentially combinatorial argument. The main computations needed for this argument are given as Lemma \ref{vanishing} and Lemma \ref{factor} below. Before giving the proof, however, we mention some immediate consequences of Theorem \ref{grr}.

In the following, we write $G(X)$ for the Grothendieck group of coherent sheaves on $X$. When $X$ is equidimensional (e.g.\ when $X$ is an integral variety) the group $G(X)$ comes equipped with a filtration $F_\tau^\bullet(X)\subset G(X)$ whose term $F^i_\tau(X)$ is defined as the subgroup of $G(X)$ generated by coherent sheaves whose support has codimension-$i$ or greater. Typically, $F_\tau^\bullet(X)$ is called either the topological filtration or the coniveau filtration of $G(X)$.

Similar to our notation with the $\gamma$-filtration, we write $\mathrm{gr}_\tau G(X)$ for the graded group associated to the topological filtration. If $X$ is a smooth variety, then the group $\mathrm{gr}_\tau G(X)$ is even a ring. This is because, for a smooth variety $X$, the group $G(X)$ is isomorphic to the group $K(X)$ via the canonical morphism $K(X)\rightarrow G(X)$ sending the class of a locally free sheaf to the class of itself. Under this isomorphism, the group $G(X)$ inherits a ring structure for which the filtration $F^\bullet_\tau(X)$ is multiplicative.

In one formulation, the Grothendieck-Riemann-Roch theorem without denominators (see \cite[Expos\'e XIV]{MR0354655} or \cite[Example 15.2.16]{MR801033}) is the statement that there exists both a surjection of graded rings $\varphi_X:\mathrm{CH}(X)\rightarrow \mathrm{gr}_\tau G(X)$ for any smooth variety $X$, where $\mathrm{CH}(X)$ is the Chow ring of cycles on $X$ modulo rational equivalence, and a homomorphism $c_i:\mathrm{gr}_\tau G(X)\rightarrow \mathrm{CH}^i(X)$ for any $i\geq 0$ induced by the $i$th Chern class $c_i^{\mathrm{CH}}:K\rightarrow \mathrm{CH}^i$, so that for $i>0$ both of the compositions on degree-$i$ summands
\[c_i\circ \varphi_X^i=(-1)^{i-1}(i-1)!\quad \mbox{and}\quad \varphi_X^i\circ c_i=(-1)^{i-1}(i-1)!\] are both multiplication by $(-1)^{i-1}(i-1)!$. Together with Theorem \ref{grr} this implies:

\begin{cor}\label{cor}
Let $X$ be a smooth and connected variety. For each integer $i\geq 0$ there exists a commuting diagram like the following one.
\[\tag{Di}\begin{tikzcd}
B^i(X)\arrow{r}{b^i_{\mathrm{CH}}}\arrow[shift left=1]{d}{b^i_\gamma}\arrow{dr}{b_{\tau}^i} & \mathrm{CH}^i(X)\arrow[shift left=1]{d}\\ \mathrm{gr}_\gamma^i K(X)\arrow{r} & \mathrm{gr}^i_\tau G(X)
\end{tikzcd}\] Moreover, these commuting diagrams have the following properties:
\begin{enumerate}[\normalfont (1)]
\item all of the arrows in the diagrams \emph{(D0)} and \emph{(D1)} are isomorphisms;
\item if $\mathrm{dim}(X)\leq 2$, then every arrow of \emph{(D2)} is an isomorphism;
\item for any $i\geq 0$, all of the arrows in \emph{(Di)} become isomorphisms after tensoring with $\mathbb{Q}$.
\end{enumerate}
\end{cor}

\begin{proof}
In the diagram (Di), the unlabeled vertical arrow is the degree-$i$ summand $\varphi^i_X$ of the surjection $\varphi_X$ from the Grothendieck-Riemann-Roch without denominators; it follows from this theorem that $\varphi_X^0$, $\varphi_X^1$, and $\varphi_X^2$ are isomorphisms while $\varphi_X^i$ has torsion kernel in general. 
	
The unlabeled horizontal arrow is induced by the isomorphism $K(X)\rightarrow G(X)$ described in the paragraphs above the corollary statement. Indeed, when one identifies the two rings under this isomorphism then there is an inclusion of filtrations $F_\gamma^i(X)\subset F_\tau^i(X)$ for all $i\geq 0$. For $i=0$ or $i=1$, the induced map on the associated graded groups is known to be an isomorphism; if $\mathrm{dim}(X)\leq 2$, then this holds also for $i=2$, see \cite[Proposition 2.14]{MR1615533}. For arbitrary $i\geq 0$, this arrow is known only to have torsion kernel and cokernel, see \cite[Exspos\'e XIV, \S4]{MR0354655}.

The arrows labeled $b^i_{\mathrm{CH}}$, $b^i_\gamma$, and $b^i_{\tau}$ are the maps guaranteed from the universal property Proposition \ref{up} of the functor $B$ (with the caveat that one restricts to the category of smooth varieties; see Remark \ref{catchall}). For each $i\leq 2$, the map $b^i_\gamma$ is an isomorphism by Theorem \ref{grr}. For arbitrary $i\geq 0$, the same theorem shows that the kernel of $b^i_\gamma$ is torsion. This implies that the morphisms $b^i_{\mathrm{CH}}$ and $b^i_\tau$ are isomorphisms in each of the cases (1)-(3) as well.
\end{proof}

\begin{rmk}
For general $i\geq 0$, the maps $b^i_{\mathrm{CH}}$ in the diagram (Di) of Corollary \ref{cor} need not be injective or surjective. If $i=2$, then it's known that the map $b^i_{\mathrm{CH}}$ is surjective but not necessarily injective. As an example of a variety where $b^2_{\mathrm{CH}}$ has nontrivial kernel one can take a Severi--Brauer variety $X=\mathbf{SB}(A)$ associated to division algebra $A$ with $\mathrm{ind}(A)=p^2$, for some odd prime $p$, that decomposes $A=A_1\otimes A_2$ into the tensor product of two division algebras $A_1,A_2$ each of index $p$. In this case $$B^2(X)=\mathrm{gr}^2_\gamma K(X)= \mathbb{Z}\oplus \mathbb{Z}/p\mathbb{Z}$$ contains nontrivial torsion because of \cite[Proposition 4.13]{MR1615533} but, $\mathrm{CH}^2(X)=\mathbb{Z}$ is torsion free by \cite[Proposition 5.3]{MR1615533}.

For an example of a smooth and projective variety $X$ where the map $b^i_{\mathrm{CH}}$ is not surjective for some $i>2$, one can take $X=G/B$ to be the variety of complete flags contained in the group $\mathrm{O}^+(n)$ for any $n\geq 7$ and with $i=\mathrm{dim}(X)$, see \cite[Expos\'e XIV, \S4, 4.5]{MR0354655}.

For an example where $b^3_\gamma$ is not injective, one can take $X$ to be a $4$-approximation of the classifying space $B\mathrm{O}^+(2n+1)$ for any $n\geq 1$. Indeed, by \cite[Theorem 5.5]{10.1093/imrn/rnz049} the map $b^i_{\mathrm{CH}}$ for this $X$ is surjective for all $i\geq 0$ (so that the $\gamma$-filtration and the topological filtration coincide, see \cite[Proposition 3.3 (2)]{10.1093/imrn/rnz049}) but the map $\varphi_X^3$ is not an isomorphism.
\end{rmk}

\begin{rmk}
	Let $X$ be a smooth and connected variety such that $\mathrm{CH}(X)$ is generated by Chern classes. In this case, Theorem \ref{grr} can be used to recover the standard Grothendieck-Riemann-Roch theorem without denominators \cite[Example 15.2.16]{MR801033} for this $X$.
	
	Indeed, if $\mathrm{CH}(X)$ is generated by Chern classes then the canonical morphism $$\mathrm{gr}_\gamma K(X)\rightarrow \mathrm{gr}_\tau G(X)$$ is an isomorphism, see \cite[Proposition 3.3]{10.1093/imrn/rnz049}. One can then check that the compositions \[c_i\circ \varphi_X^i\quad \mbox{and}\quad \varphi_X^i\circ c_i\] are both multiplication by $(-1)^{i-1}(i-1)!$ by going around the outside of the commutative square below
	\[\begin{tikzcd}
	B^i(X)\arrow{r}{b^i_{\mathrm{CH}}}\arrow[shift left=1]{d}{b^i_\gamma} & \mathrm{CH}^i(X)\arrow[shift left=1]{d}{\varphi_X^i}\\ \mathrm{gr}_\gamma^i K(X)\arrow[equals]{r}\arrow[shift left=1]{u}{c_i^B} & \mathrm{gr}^i_\tau G(X)\arrow[shift left=1]{u}{c_i}
	\end{tikzcd}\] and applying Theorem \ref{grr}
\end{rmk}

Now we return to the proof of Theorem \ref{grr}.

\begin{lem}\label{vanishing}
Let $\mathcal{L}_1,...,\mathcal{L}_{i+1}$ be $i+1$ invertible sheaves on some scheme $P$ which can be realized as a chain of projective bundles over $X$. Then $$c_k^B\left(\prod_{j=1}^{i+1} (\mathcal{L}_j-1)\right)=0$$ inside of $B(P)$ for all $0<k\leq i$.
\end{lem}

\begin{proof}
We proceed by induction on the length $i+1$ of the product. For our base case, we observe that \begin{multline*}c_1^B((\mathcal{L}_1-1)(\mathcal{L}_2-1))=c_1^B(\mathcal{L}_1\otimes \mathcal{L}_2-\mathcal{L}_1-\mathcal{L}_2+1)\\=c_1^B(\mathcal{L}_1)+c_1^B(\mathcal{L}_2)-c_1^B(\mathcal{L}_2)-c_1^B(\mathcal{L}_1)+c_1^B(1)=0.\end{multline*} For our induction hypothesis, we assume the Chern class $c_k^B$ of any product of such elements of length $i$ vanishes for all $k\leq i-1$. Let $\prod_{j=1}^i (\mathcal{L}_j-1)=[\mathcal{F}]-[\mathcal{G}]$. Then
\begin{align*}
c_t^B\left(\prod_{j=1}^{i+1} (\mathcal{L}_{j}-1)\right)&= c_t^B((\mathcal{F}-\mathcal{G})(\mathcal{L}_{i+1}-1))\\
&=\frac{c_t^B(\mathcal{F}\otimes \mathcal{L}_{i+1})}{c_t^B(\mathcal{G}\otimes \mathcal{L}_{i+1})c_t^B(\mathcal{F}-\mathcal{G})}\\
&=\frac{c^B_\tau(\mathcal{F})}{c^B_\tau(\mathcal{G})c_t^B(\mathcal{F}-\mathcal{G})} \tag*{(using $\tau=\frac{t}{1+c_1^B(\mathcal{L}_{i+1})t}$, c.f.\ Example \ref{line})}\\
&=\frac{c^B_\tau(\mathcal{F}-\mathcal{G})}{c_t^B(\mathcal{F}-\mathcal{G})}\\
&= \frac{1+c^B_i(\mathcal{F}-\mathcal{G})\tau^i + c^B_{i+1}(\mathcal{F}-\mathcal{G})\tau^{i+1}+\cdots}{1+c^B_i(\mathcal{F}-\mathcal{G})t^i+c^B_{i+1}(\mathcal{F}-\mathcal{G})t^{i+1}+\cdots}\tag*{(by induction hypothesis)}\\
&=1-ic^B_{i}(\mathcal{F}-\mathcal{G})c^B_1(\mathcal{L}_{i+1})t^{i+1}+\cdots
\end{align*} as claimed.
\end{proof}

\begin{lem}\label{factor}
Let $\mathcal{L}_1,...,\mathcal{L}_{i}$ be $i$ invertible sheaves on some scheme $P$ which can be realized as a chain of projective bundles over $X$. Then $$c_i^B\left( \prod_{j=1}^i (\mathcal{L}_j-1)\right)=(-1)^{i-1}(i-1)!\prod_{j=1}^i c_1^B(\mathcal{L}_j)$$ inside of $B(P)$.
\end{lem}

\begin{proof}
Expand the term $c^B_i(\mathcal{F}-\mathcal{G})$ in the last expression of the proof of Lemma \ref{vanishing}.
\end{proof}

\begin{proof}[Proof of Theorem \ref{grr}]
Let $x\in F_\gamma^i(X)$ be an element of the $i$th piece of the $\gamma$-filtration on $K(X)$ for some $i\geq 1$. The proof will be complete if we can show that there is a scheme $P$, that can be realized as a chain of projective bundles over $X$, with the property that the pullback of $x$ to $K(P)$ can be written as a sum or difference of monomials of the form $(\mathcal{L}_1-1)\cdots (\mathcal{L}_j-1)$ for some $j\geq i$. Indeed, assuming that this is the case, there is a commuting square
\[\begin{tikzcd}F_\gamma^i(X)\arrow{r}\arrow{d}{c_i^B} & K(P)\arrow{d}{c_i^B}\\ B^i(X)\arrow{r} & B^i(P) \end{tikzcd}\] where the horizontal pullback morphisms are injections. We find $$c_t^B(x)= c_t^B\left(\sum_{m=0}^k\left(\pm\prod_{j=1}^{n_m} (\mathcal{L}_{m_j}-1)\right)\right)=\prod_{m=0}^k c_t^B(\prod_{j=1}^{n_m}(\mathcal{L}_{m_j}-1))^{\pm 1}.$$ The latter factors vanish whenever $n_m> i$ by Lemma \ref{vanishing} while the latter factors are equal $$1+(-1)^{i-1}(i-1)!\left(\pm\prod_{j=1}^{n_m} c_1^B(\mathcal{L}_{m_j})\right)t^i+\cdots$$ whenever $n_m=i$ by Lemma \ref{factor}. Since $$b^i_\gamma(c_1^B(\mathcal{L}_j))=b^i_\gamma(-c_1^B(\mathcal{L}_j^\vee))=-b^i_\gamma(c_1^B(\mathcal{L}_j^\vee))=\mathcal{L}_j-1,$$ where we use Remark \ref{dual} for the first equality, the proof is completed once we can show our starting assumption.

To do this, we start by writing $$x=\sum_{m=0}^k\left(\pm\prod_{j=1}^{n_m} \gamma^{m_j}(x_{m_j})\right)$$ for some elements $x_{m_j}$ in $F_\gamma^1(X)$. Note that we can focus on a single monomial since, if we prove that a monomial can be written in the desired way then the same fact follows for the sum. So assume that there's an expression $x=\gamma^{n_1}(x_1)\cdots \gamma^{n_j}(x_j)$ for some $n_1+\cdots + n_j \geq i$. Each $x_k$, belonging to $F_\gamma^1(X)$, can be written as $$x_k=[\mathcal{F}]-[\mathcal{G}]=[\mathcal{F}]-\mathrm{rk}(\mathcal{F})-([\mathcal{G}]-\mathrm{rk}(\mathcal{G}))$$ for some $\mathcal{F},\mathcal{G}$ that depend on $k$.

Now there is a scheme $P$ which can be realized as a chain of projective bundles over $X$ such that each of the $\mathcal{F},\mathcal{G}$'s can be written $$x_k=[\mathcal{F}]-\mathrm{rk}(\mathcal{F})-([\mathcal{G}]-\mathrm{rk}(\mathcal{G}))=(\mathcal{L}_1+\cdots +\mathcal{L}_n-n)-(\mathcal{L}'_1+\cdots + \mathcal{L}'_n-n)$$ with the $(\mathcal{L})$'s and $(\mathcal{L}')$'s depending on $k$ still. Another way to say this is that we can find such a $P$ so that, for every $k$, we have an expression like $$x_k=(\mathcal{L}_1-1)+\cdots + (\mathcal{L}_n-1) -(\mathcal{L}'_1-1)-\cdots - (\mathcal{L}'_n-1).$$ Finally, applying the operation $\gamma_t=\sum_{j\geq 0}\gamma^j t^j$ we find
\begin{align*}\gamma_t(x_k) &=\gamma_t\left(\sum_{j=1}^n (\mathcal{L}_j-1)-\sum_{j=1}^n(\mathcal{L}_j'-1)\right)\\
& =\frac{\gamma_t\left(\sum_{j=1}^n (\mathcal{L}_j-1)\right)}{\gamma_t\left(\sum_{j=1}^n(\mathcal{L}_j'-1)\right)}\\
&=\frac{\sum_{j\geq 0} \sigma_j t^j}{\sum_{j\geq 0} \sigma'_j t^j}
\end{align*}
where $\sigma_j$ is the $j$th elementary symmetric polynomial in the variables $(\mathcal{L}_1-1),...,(\mathcal{L}_n-1)$ and similarly for $\sigma'_j$ with $(\mathcal{L}'_1-1),...,(\mathcal{L}'_n-1)$. Expanding this series in $t$ we find that $\gamma^{m_k}(x_k)$ is a polynomial, homogeneous and symmetric in variables like $(\mathcal{L}-1)$, of degree $m_k$. But, this completes the proof since we've shown that there is a scheme $P$ which can be realized as a chain of projective bundles over $X$ such that $x$ can be written in the desired form.
\end{proof}
\bibliographystyle{amsalpha}
\bibliography{bib3}
\end{document}